\documentclass[11pt]{amsart}
  
\usepackage{graphicx}
\graphicspath{ {images/} }
\usepackage{amssymb,verbatim,amscd,amsfonts,amsmath,amsthm,enumerate}
\usepackage[all]{xy}
\usepackage{subfig}
\usepackage{tikz}
\usetikzlibrary{shapes,arrows,shadows}
\usetikzlibrary{decorations.markings}
\usepackage{color}
\usepackage[normalem]{ulem}

\usepackage{setspace}

\newtheorem*{namedtheorem}{\theoremname}
\newcommand{\theoremname}{testing}
\newenvironment{named}[1]{\renewcommand{\theoremname}{#1}\begin{namedtheorem}}{\end{namedtheorem}}

\usepackage{hyperref}
\usepackage{wrapfig}
\usetikzlibrary{arrows}
\usepackage{pinlabel}

\long\def\symbolfootnote[#1]#2{\begingroup%
\def\thefootnote{\fnsymbol{footnote}}\footnote[#1]{#2}\endgroup}

\newtheorem{theorem}{Theorem}[section]
\newtheorem*{NoNumberTheorem}{Theorem}
\newtheorem{proposition}[theorem]{Proposition}
\newtheorem{corollary}[theorem]{Corollary}
\newtheorem{lemma}[theorem]{Lemma}

\newtheorem{question}[theorem]{Question}

\theoremstyle{definition}

\newtheorem{remark}[theorem]{Remark}

\newtheorem{definition}[theorem]{Definition}

\newtheorem{example}[theorem]{Example}

\newcommand{\R}{\mathbb{R}}

\newcommand{\N}{\mathbb{N}}
\newcommand{\Z}{\mathbb{Z}}

\newcommand{\Cc}{\mathcal{C}}
\newcommand{\Bc}{\mathcal{B}}

\newcommand{\Bh}{\mathcal{H}}

\newcommand{\Homeo}{\rm Homeo}

\newcommand{\Pc}{\mathcal{P}}

\def\Aut{\operatorname{Aut}}
\def\Out{\operatorname{Out}}
\def\Mod{\operatorname{Mod}}
\def\Homeo{\operatorname{Homeo}}

\def\PMod{\operatorname{PMod}}
\def\Stab{\operatorname{Stab}}

\def\GL{\operatorname{GL}}
\def\inte{\operatorname{int}}

\begin{document}
\title[Asymptotic mapping class groups]{Asymptotic mapping class groups of closed surfaces punctured along Cantor sets}

\author{Javier Aramayona}
\address{Universidad Aut\'onoma de Madrid \& ICMAT \\ C. U. de Cantoblanco. 28049, Madrid, Spain}
\email{aramayona@gmail.com}
 \author{Louis Funar} 
 \address{ Institut Fourier, UMR 5582, Laboratoire de Math\'ematiques \\
Universit\'e Grenoble Alpes, CS 40700, 38058 Grenoble cedex 9, France}
\email{louis.funar@univ-grenoble-alpes.fr}

\date{\today}
\thanks{The first author was partially funded by grants RYC-2013-13008 and MTM2015-67781. He also acknowledges support from U.S. National Science Foundation grants DMS
1107452, 1107263, 1107367 "RNMS: Geometric Structures and Representation Varieties" (the
GEAR Network). This project has received funding from the European Union’s Horizon 2020 research and innovation program under the Marie Skłodowska-Curie grant agreement No 777822.}

%
%

\begin{abstract}
We introduce subgroups $\Bc_g< \Bh_g$ of the mapping class group $\Mod(\Sigma_g)$ of a closed surface of genus $g \ge 0$ with a Cantor set removed, which are extensions of Thompson's group $V$ by a direct limit of mapping class groups of compact surfaces of genus $g$. 

We first show that both $\Bc_g$ and $\Bh_g$  are finitely presented, and that $\Bh_g$ is dense in $\Mod(\Sigma_g)$. 
We  then exploit the relation with Thompson's groups to study properties $\Bc_g$ and $\Bh_g$  in analogy with known facts about finite-type mapping class groups. For instance, their  homology coincides with the stable  homology of the mapping class group of genus $g$, every automorphism is geometric, and every homomorphism from a higher-rank lattice has finite image.

In addition, the same connection with Thompson's groups will also prove that $\Bc_g$ and $\Bh_g$ are  not linear and do not have Kazhdan's Property (T), which represents a  departure from the current knowledge about finite-type mapping class groups.

%
%
%
%
%
\end{abstract}

\maketitle

\section{Introduction}  
There has been a recent surge of activity around mapping class groups of infinite-type surfaces, namely those whose fundamental group is not finitely generated. These groups share many properties with their finite-type counterparts (e.g. \cite{BDR,HMV}), but also show rather different behaviour (e.g. \cite{APV,PV}). 

Here we will focus our attention on the mapping class group $\Mod(\Sigma_g)$ of the surface $\Sigma_g$, namely the closed orientable surface of genus $g\ge 0$ with a Cantor set $C$ removed. This group is related to the homeomorphism group $\Homeo(C)$ through the short exact sequence (see Section \ref{sec:V} for details):
\begin{equation}
1 \to \PMod(\Sigma_g) \to \Mod(\Sigma_g) \to \Homeo(C) \to 1,
\label{eq:SESgen}
\end{equation}
where $\PMod(\Sigma_g)$ is the {\em pure} mapping class group, namely the subgroup of $\Mod(\Sigma_g)$ whose elements fix $C$ pointwise. 

 In this article we study two countable subgroups $\Bc_g < \Bh_g < \Mod(\Sigma_g)$, whose elements {\em asymptotically preserve a rigid structure} on $\Sigma_g$. We remark that $\Bc_0$ and $\Bh_0$ were previously introduced in \cite{FK,FN} under the names $\Bc$ and $\Bc^{1/2}$, respectively. We now give a brief description of these groups, referring the reader to Section \ref{sec:defs} for a complete definition.  
 
 
 We first need to introduce the notion of a {\em rigid structure}; for this purpose, it will be convenient to start with the genus-zero case.  The reader should keep  Figure \ref{fig:genus0} in mind. 
  A rigid structure on $\Sigma_0$ is a triple $(P,\mathcal A, \Sigma_0^+)$, where:
  \begin{itemize}
  \item $P$ is a pants decomposition of $\Sigma_0$,
  \item $\mathcal A$ is a set of pairwise-disjoint properly embedded arcs such that   (the closure of) every connected component of $S- P$, referred to as a {\em pair of pants},  is intersected by exactly three arcs in $\mathcal{A}$, and
  \item $\Sigma_0^+$ is a choice of one of the two connected components of $\Sigma_0 - \bigcup_{a\in \mathcal{A}} a$.
\end{itemize}

\begin{figure}[htb]
\begin{center}
\includegraphics[width=3in,height=2in]{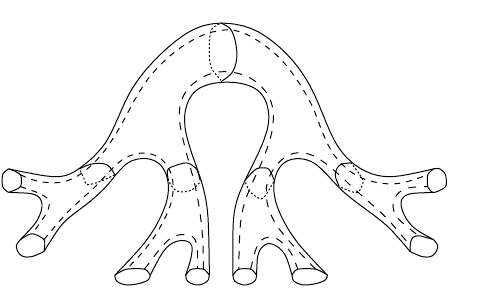} \caption{The rigid structure on $\Sigma_0$.} \label{fig:genus0}
\end{center}
\end{figure}

Let $\Sigma_0^1$ be the result of puncturing $\Sigma_0$ once. A rigid structure on $\Sigma_0^1$ is a rigid structure on the surface (homeomorphic to $\Sigma_0$)  obtained by filling in the isolated puncture of $\Sigma_0^1$. 
 
 A rigid structure on $\Sigma_g$ ($g \ge 1$) consists of a simple closed curve $\alpha$ which cuts off a once-puncture surface of genus $g$, together with a rigid structure for the component of $\Sigma_g - \alpha$ homeomorphic to $\Sigma_0^1$. 
 
 \smallskip

\noindent{\bf The group $\Bc_g$.} Fix, once and for all, a rigid structure on $\Sigma_g$. A homeomorphism $f:\Sigma_g \to \Sigma_g$ is {\em asymptotically rigid} if it preserves  the rigid structure outside some compact subsurface $Z(f)$ of $\Sigma_g$; see Section \ref{sec:rigid} for a complete definition. We define  $\Bc_g$ as the subgroup of $\Mod(\Sigma_g)$ whose elements have at least one asymptotically-rigid representative. Obviously, every element of $\Mod(\Sigma_g)$ with compact support belongs to $\Bc_g$, although the converse is not true. Indeed, denote by $\PMod_c(\Sigma_g)$ the subgroup of $\PMod(\Sigma_g)$ generated by compactly-supported elements, which is a direct limit of mapping class groups of compact genus-$g$ subsurfaces of $\Sigma_g$.  In Proposition \ref{prop:sesB} we will generalize a result of \cite{FK} to prove that the sequence \eqref{eq:SESgen}, when restricted to $\Bc_g$, reads:
\begin{equation}
1 \to \PMod_c(\Sigma_g) \to \Bc_g \to V \to 1,
\end{equation} 
where the rightmost non-trivial group is {\em Thompson's group} $V$ (see e.g. \cite{CFP}).  This sequence reveals a fascinating connection between mapping class groups and Thompson's groups, and will be a key ingredient in the proofs of most of our results.

\begin{remark}
We remark that the notation $\PMod_c(\Sigma)$ is somewhat redundant, for if an element of $\Mod(\Sigma)$ has compact support then it automatically belongs to $\PMod(\Sigma)$. However, we will use the notation $\PMod_c(\Sigma)$ to emphasize the connection with the sequence \eqref{eq:SESgen}
\end{remark}

\smallskip

\noindent{\bf The group $\Bh_g$.} The group $\Bh_g$ is defined in a similar fashion: its elements are those mapping classes which have a representative which preserve $(P,\mathcal A)$ outside some compact subsurface. Observe that $\Bc_g < \Bh_g$, although the inclusion is strict: for instance, a {\em half-twist} lies in $\Bh_g \setminus \Bc_g$. For this reason, the group $\Bh_g$ is sometimes referred to as the {\em group of half-twists} \cite{FN}. Using the same reasoning as above, equation \eqref{eq:SESgen} restricts to a short exact sequence 
\begin{equation}
1 \to \PMod_c(\Sigma_g) \to \Bh_g \to V_2[\mathbb{Z}_2] \to 1,
\end{equation}
where $V_2[\mathbb{Z}_2]$ is the {\em Higman-Thompson group} $V_2[\mathbb{Z}_2]$ (see \cite{BDJ} and Section \ref{sec:V}). A surprising result of Bleak-Donoven-Jonu\v{s}as \cite{BDJ} asserts that $V$ and $V_2[\mathbb{Z}_2]$ are conjugate as subgroups of $\Homeo(C)$ through an explicit homeomorphism of $C$ (a {\em cellular automaton}). 

%

A large part of the motivation for considering $\Bh_g$ comes from the study of smooth mapping class groups. Put a differentiable structure on the closed orientable surface $S_g$ of genus $g$, and realize $C$ as the the middle-third Cantor set on a smoothly-embedded interval on $S_g$. Let $\Mod^s(S_g, C)$ denote the {\em smooth mapping class group} of the pair $(S_g,C)$, namely the group of isotopy classes of smooth diffeomorphisms of $S_g$ preserving globally the Cantor set $C$. The following is a recent result of Neretin and the second author \cite{FNe}: 

\begin{NoNumberTheorem}[\cite{FNe}, Cor. 2]
For every $g \ge 0$, we have $\Bh_g \cong \Mod^s(S_g,C)$.
\end{NoNumberTheorem}

In particular, $\Mod^s(S_g,C)$ is countable; observe that, in stark contrast, the ``topological'' mapping class group $\Mod(\Sigma_g)$ is uncountable. 

\subsection{Results} As we will see, the topological restrictions on the elements of $\Bc_g$ and $\Bh_g$ impose strong finiteness conditions on the groups. More concretely, we will prove:

\begin{theorem}
For every $g\ge 0$,   $\Bc_g$ and $\Bh_g$ are finitely presented. 
\label{thm:fg}
\end{theorem}

We remark that the case $g=0$ of the above theorem was proved by Kapoudjian and the second author in \cite{FK}; in fact, it will serve as the base case for the inductive argument behind the proof of Theorem \ref{thm:fg}. 

In spite of the above result, we will prove that $\Bh_g$ serves as a good approximation for the mapping class group:

\begin{theorem}
For every $g\ge 0$, $\Bh_g$ is dense in $\Mod(\Sigma_g)$.
\label{thm:limit}
\end{theorem}

This theorem should be compared with a recent result of Patel-Vlamis \cite{PV} which asserts that $\PMod_c(\Sigma_g)$ is dense in $\PMod(\Sigma_g)$

\medskip

Next, we turn our attention to the study of properties of the groups $\Bc_g$ and $\Bh_g$, through the comparison with known/expected/desired properties of mapping class groups of finite-type surfaces.

\subsubsection{Homological stability}
Let $S_{g,n}$ denote the surface of genus $g$ with $n$ boundary components. A celebrated result of Harer \cite{Harer} asserts that, for a fixed genus $g\ge 3$, the $k$-th homology group of the mapping class group $\PMod(S_{g,n})$ is independent of $n$, provided that $k$ is ``sufficiently small'' with respect to $g$ (by a result of Boldsen \cite{Boldsen}, $k \le 2g/ 3$ suffices). For this reason, this homology group is called the $k$-th {\em stable homology group} of the mapping class group of genus $g$. 
Using a translation of the proof of \cite[Theorem 3.1]{FK2} to our setting, we will show: 

\begin{theorem}
For every $k \le 2g/3$,  $H_k(\Bc_{g}, \Z)$ and $H_k(\Bh_{g}, \Z)$ are isomorphic to the $k$-th stable homology group of the mapping class group of genus $g$. 
\label{thm:homology}
\end{theorem} 

Powell \cite{Powell} proved that $\Mod(S_g)$ is {\em perfect}, i.e. has trivial abelianization. As a 
consequence of  the proof of Theorem \ref{thm:homology}, we have:

\begin{corollary}\label{cor:perfect}
$\Bc_g$ and $\Bh_g$ are perfect for every $g\ge 3$. 
\end{corollary}

%

\subsubsection{Isomorphic classification} 
With a few well-understood exceptions,
mapping class groups of finite-type surfaces are isomorphic if and only if the underlying surfaces are homeomorphic. To see this, one may compare the  virtual cohomological dimension of the mapping class group \cite{Harer-vcd} with the maximal rank of a free-abelian subgroup \cite{BLM}. In the case of the groups $\Bc_g$ and $\Bh_g$ both these quantities are infinite.  However, we will prove:

\begin{theorem}
If $0\le g<h$ and $2\leq h$, then there are no surjective homomorphisms $\Bc_h \to \Bc_g$. 
\label{thm:Ainj}
\end{theorem}

As will become transparent, the same argument will yield that there are no surjective maps $\Bh_h \to \Bh_g$ (resp. $\Bh_h \to \Bc_g$, and $\Bc_h \to \Bh_g$). 
As a consequence of Theorem \ref{thm:Ainj}, $\Bc_g \cong \Bc_h$ (resp. $\Bh_g \cong \Bh_h$) if and only if $g = h$. In light of this, an obvious question is: 

\begin{question}
Are $\Bc_g$ and $\Bh_g$ isomorphic? 
\end{question}

We stress that, although the results from \cite{BDJ} might suggest a  positive answer to the question above, the answer remains unknown for all values of $g$. In a more general situation, however, if we replace the binary Cantor set $C$ by the set of ends of a regular tree of valence higher than $3$ then the corresponding Thompson groups groups $V$ and $V[\Z_2]$ might be non-isomorphic; compare with \cite{BDJ}.

\subsubsection{Rigidity} A celebrated theorem of Ivanov \cite{Ivanov} states that the mapping class group of a (sufficiently complicated) finite-type surface is {\em rigid}: every automorphism is induced by a surface homeomorphism. This has recently been extended to the infinite-type setting by Patel-Vlamis \cite{PV} and Bavard-Dowdall-Rafi \cite{BDR}.
Along similar lines, our next result asserts that $\Bc_g$ and $\Bh_g$ are also rigid. Given a group $G$ and a subgroup $H$, write $\Aut(G)$ for the automorphism group of $G$, and denote by $N_G(H)$ the normalizer of $H$ in $G$. We have:

\begin{theorem}
For every $g\ge 0$, $\Aut(\Bc_g) \cong N_{\Mod(\Sigma_g)}(\Bc_g)$ and $\Aut(\Bh_g) \cong N_{\Mod(\Sigma_g)}(\Bh_g)$. 
\label{thm:autos}
\end{theorem}

An immediate consequence of Ivanov's theorem in the finite-type case is that the outer automorphism group of $\Mod(S_{g,n})$ is finite; in fact, it is trivial for all but finitely many surfaces. However, in Corollary \ref{cor:out} we will see that this is no longer true for the groups $\Bc_g$ and $\Bh_g$. Denote by $\Out(G)$ the outer automorphism group of a group $G$, that is, the group of conjugacy classes of automorphisms of $G$. We will prove: 

\begin{corollary}
For every $g\ge 0$, $\Out(\Bc_g)$ and $\Out(\Bh_g)$ are  infinite. 
\label{cor:out}
\end{corollary}

\begin{example}
Let $P$ be the pants decomposition underlying the rigid structure on 
$\Sigma_0^1$, and consider the element $t_{\infty} \in \Mod(\Sigma_0^1)$ obtained as the product of all (say left) half-twists about the curves of $P$. 
We further embed $\Sigma_0^1$ in $\Sigma_g$ in such way that the pants decomposition $P$ underlying the rigid structure of $\Sigma_0^1$ is sent to the pants decomposition underlying the rigid structure of $\Sigma_g$. We will show later (see Lemma \ref{lem:infkernel}) that the image of $t_\infty$ under the homomorphism induced by the embedding $\Sigma_0^1 \to \Sigma_g$ produces an infinite-order element 
of $N_{\Mod(\Sigma_g)}(\Bc_g)$ lying in the kernel of the homomorphism $\Out(\Bc_g)\to \Out(V)$.
\end{example}\label{example}

\subsubsection{Homomorphisms from lattices} Building up on work of Ivanov \cite{Ivanov-lattice} and Kaimanovich-Masur \cite{KM}, Farb-Masur \cite{FM} proved that any homomorphism from a higher-rank lattice to a finite-type mapping class group has finite image (see also \cite{BF,H} for different proofs of this result). Using this result, we will prove:

\begin{theorem}
Let $\Gamma$ be a lattice in a  semisimple Lie group $G$ of real rank at least two, where $G$ has no compact factors isogenous to ${\rm SU}(1,n)$ or ${\rm SO}(1,n)$. 
For every $g\ge 0$, any homomorphism from $\Gamma$ to $\Bc_g$ or $\Bh_g$ has finite image. 
\label{thm:lattice}
\end{theorem}

\subsubsection{Kazhdan's Property (T)} A compactly generated group has {\em Kazhdan's Property (T)} if every unitary representation that has almost invariant vectors also has an invariant vector. Since we will not need any further details about Property (T), we simply refer the reader to the book \cite{BHV} for details. 
It is expected \cite{A} that mapping class groups of finite-type surfaces do not have Property (T). Using similar arguments to the ones in the proof of Theorem \ref{thm:lattice}, we will observe: 

\begin{theorem}
 $\Bc_g$ and $\Bh_g$ do not have Kazhdan's Property (T) for any $g\ge 0$. 
\label{thm:T}
\end{theorem}

\subsubsection{Non-linearity} A well-known open question asks  whether finite-type mapping class groups are  linear. The only known result in this direction is a theorem of Bigelow-Budney \cite{Bigelow}, who proved that the mapping class group of the closed surface of genus two is linear, by exploiting its relation with braid groups. 
In sharp contrast, in Proposition \ref{split} we will see that $\Bc_g$ contains an isomorphic copy of Thompson's group $F$, for all $g\ge 0$. Combining this with Theorem \ref{thm:fg}, we have: 

\begin{theorem}
$\Bc_g$ is not linear for any $g\ge 0$. In particular, $\Bh_g$ is not linear either. 
\label{thm:nonlinear}
\end{theorem}



\medskip

\noindent{\bf Plan of the paper.}  In Section \ref{sec:defs} we give some basic definitions and set the notation used in the rest of the paper. In Section \ref{sec:rigid} we will define the groups $\Bc_g$ and $\Bh_g$. In Section \ref{sec:V} we will explore their relation to  Thompson's group $V$, which will be a key ingredient in the proof of many of our main results. We then proceed to prove the results mentioned in the introduction. In this direction, Theorem \ref{thm:fg} will be proved in Section \ref{sec:proofT1}, and Theorem \ref{thm:limit} in Section \ref{sec:dense}.  Section  \ref{sec:homology}. will deal with the proof of Theorem \ref{thm:homology}. In Section \ref{sec:injective} we will prove Theorem \ref{thm:Ainj}, while Section \ref{sec:proofaut} is concerned with the proof of Theorem \ref{thm:autos}. In Section \ref{sec:lattices} we will establish Theorems \ref{thm:lattice} and \ref{thm:T}.

\medskip

\noindent{\bf Acknowledgements.} We thank Y. Antol\'in,  J. Bavard, L. Bowen, J. Hern\'andez, C. Kapoudjian, T. Koberda, C. Mart\'inez-P\'erez, H. Parlier, P. Patel, J. Souto, and N. Vlamis for conversations. We are indebted to Juan Souto for suggesting the proof of Lemma \ref{lem:twists2twists} and to the referee for carefully reading our paper and his numerous comments and corrections. Parts of this paper were written while the first author was visiting Yale University, to which he is grateful for its hospitality.
\section{Preliminaries}
\label{sec:defs} 

In this section we recall some of the basic definitions about surfaces and their mapping class groups. 

\subsection{Curves and surfaces} Let $S$ be a connected orientable surface, of finite or infinite topological type. If $S$ has punctures, we will regard them either as marked points on $S$ or as topological ends of $S$, and we will feel free to switch between the two viewpoints without any further mention. We will denote by $S_{g,n}^p$ the compact surface of genus $g\ge 0$ with $n\ge 0$ boundary components and $p\ge 0$ punctures. Similarly, let $\Sigma_{g,n}^p$ the closed surface of genus $g\ge 0$ with a Cantor set removed, with $n\ge 0$ boundary components and $p\ge 0$ isolated punctures. If either $n$ or $p$ is equal to zero we will simply omit it from the notation.  

 By a {\em curve} on $S$ we mean the isotopy class of a simple closed curve on $S$ which does not bound a disk, a once-punctured disk, or an annulus whose other boundary curve is a boundary component of $S$. We say that two curves are {\em disjoint} if they have disjoint representatives on $S$; otherwise we say that they intersect.  Given curves $a, b\subset S$, we define their {\em intersection number} $i(a,b)$ as the minimal number of points of intersection  between representatives. 
Note that $i(a,b)$ is always finite, even on a surface of infinite topological type, as curves are compact.
 
 A {\em multicurve} on $S$ is a set of pairwise disjoint curves on $S$. We say that a multicurve $M$ is {\em locally finite} if any compact subsurface of $S$ intersects only finitely many elements of $M$.  A {\em pants decomposition} of $S$ is a locally-finite multicurve that is maximal with respect to inclusion; as such, its complement on $S$ is a disjoint union of 3-holed spheres, or {\em pairs of pants}. 
 
 Finally, an {\em arc} on $S$ is a non-trivial isotopy class of properly embedded arcs on $S$.

\subsection{Mapping class group} Let $S$ be a connected orientable surface, possibly of infinite type. The {\em mapping class group} $\Mod(S)$ is the group of isotopy classes of self-homeomorphisms of $S$, where homeomorphisms and isotopies are required to fix the boundary of $S$ pointwise. 
We record the following immediate observation for further use: 

\begin{lemma}
For $g\ge 0$ and $n\ge 1$,  $\Mod(S_{g,n})<\Mod(\Sigma_g)$.
\end{lemma}

In what follows, we will need to make use of the following further subgroups of $\Mod(S)$. The {\em pure mapping class group}  $\PMod(S)$ is the subgroup of $\Mod(S)$ whose elements fix every topological end of $S$. The {\em compactly supported} mapping class group $\PMod_c(\Sigma_g)$ is the subgroup of $\PMod(\Sigma_g)$ whose elements are the identity outside a compact subsurface of $\Sigma_g$. The following is an easy observation:

\begin{lemma} Let $g\ge 0$. Consider any family $\{S_i\}$ of compact subsurfaces of $\Sigma_g$ whose union equals $\Sigma_g$, partially ordered with respect to inclusion. Then 
$\PMod_c(\Sigma_g) \cong \varinjlim (\PMod(S_i)).$
\label{lem:limit}
\end{lemma}




\section{Asymptotic mapping class groups} 
\label{sec:rigid} 
In this section we define the groups $\Bc_g$ and $\Bh_g$. 
We start by introducing the notion of rigid structure, which appeared originally in \cite{FK}. 

\subsection{Rigid structures} As in the introduction, it will be convenient to start with the genus-zero case. 
The reader should keep Figure \ref{fig:genus0} in mind. A {\em rigid structure} on $\Sigma_0$ is a triple $(P,A,\Sigma_0^+)$, where:
\begin{itemize}
\item $P$ is a pants decomposition of $\Sigma_0$, called the  {\em pants decomposition  underlying the rigid structure},
\item $A$ is a set of pairwise-disjoint arcs on $\Sigma_0$ such that for every pair of pants $Y$ of $P$ (that is, the closure of a  connected component of $\Sigma_0 - P$), there are exactly three elements of $A$ intersecting $Y$ essentially, each  connecting a different pair of boundary curves of $Y$,
\item $\Sigma_0^+$ is one of the two connected components of $\Sigma_0 - \bigcup_{a \in A} a$, called the {\em visible side} of $\Sigma_0$. 
\end{itemize}

Observe that, up to the action of $\Mod(\Sigma_0)$, there is only one rigid structure on $\Sigma_0$. 

 It will be useful  to extend the definition of  rigid structure to the surface $\Sigma_0^p$ obtained from $\Sigma_0$ by removing a finite collection of $p\ge 1$  points.  In this case, we define a rigid structure on  $\Sigma_0^p$ as a rigid structure on the surface  (homeomorphic to $\Sigma_0$) obtained from $\Sigma_0^p$ by filling every isolated puncture, subject to the condition that every isolated puncture of $\Sigma_0^p$ is contained in the same connected component of $\Sigma_0 - P$, where $P$ is the pants decomposition underlying the rigid structure of $\Sigma_0$. 

Finally, consider the case of $1\le g<\infty$. A {\em rigid structure} on $\Sigma_g$  consists of a curve $c \subset \Sigma_g$ that cuts off a disk containing every puncture of $\Sigma_g$, together with a rigid structure for the planar component of $\Sigma_g - c$, namely the one  homeomorphic  to $\Sigma_0^1$.

\subsection{The groups $\Bc_g$ and $\Bh_g$} We now define the groups $\Bc_g$ and $\Bh_g$. As mentioned above, we stress that the case $g=0$ was previously introduced  in \cite{FK,FN}.
Fix, once and for all, a rigid structure on $\Sigma_g$, and write $P$ for   the pants decomposition underlying it. 
We say that a compact subsurface $Z \subset \Sigma_g$ is {\em $P$-suited} if 
$\partial Z \subset P$, namely each boundary curve of $Z$ is an element of the pants decomposition $P$.

\begin{definition}
A homeomorphism $f:\Sigma_g \to \Sigma_g$ is {\em asymptotically rigid} if there exists a $P$--suited genus-$g$ subsurface $Z \subset \Sigma_g$ with $f(Z)$ also $P$--suited,  and such that the restriction homeomorphism $f: \Sigma_g - Z \to \Sigma_g - f(Z)$ sends: 
\begin{itemize}
\item $P \cap (\Sigma_g- Z)$ to  $P \cap (\Sigma_g- f(Z))$,
\item $A \cap (\Sigma_g- Z)$ to  $A \cap (\Sigma_g- f(Z))$,
\item The visible side of $\Sigma_g-Z$ to the visible side of $\Sigma_g-f(Z)$. 
\end{itemize}
If we drop the last requirement we say that $f$ is {\em asymptotically quasi-rigid}.
\label{def:asrig1}
\end{definition}

\noindent{\bf{Notation.}} Given an asymptotically rigid (resp. quasi-rigid) homeomorphism $f$ as above, we will refer to the subsurface $Z$ in the definition above as a {\em defining subsurface} for $f$.

\begin{remark}
Observe that if $f:\Sigma_g \to \Sigma_g$ is asymptotically (quasi-) rigid and $Z$ is a defining surface for $f$, then any $P$-suited surface containing $Z$ is also a defining surface for $f$. This observation will be heavily used in the rest of the paper, without further mention. 
\end{remark}

\medskip

We are finally in a position to define the groups we are interested in: 

\begin{definition}
Let $g,p\ge 0$. We define $\Bc_g$ (resp. $\Bc_0^p$) as the  subgroup of $\Mod(\Sigma_g)$ (resp. $\Mod(\Sigma_0^p)$) consisting of those elements that have at least one asymptotically rigid representative. 

In turn, the group $\Bh_g$ is the subgroup of $\Mod(\Sigma_g)$ consisting of those elements that have at least one asymptotically quasi-rigid representative. 
\end{definition} 

Observe that $\PMod_c(\Sigma)\subset \Bc_g$, by definition. However, the inclusion is proper, as a general element of $\Bc_g$ may permute the components of the complement of every defining subsurface.

%
\begin{remark}
The group of mapping classes of asymptotically quasi-rigid homeomorphisms of the disk punctured along a Cantor set coincides with the braided Thompson group considered by Brin \cite{Brin} and Dehornoy \cite{Deh}. 
The subgroup $BV$ of mapping classes of asymptotically rigid homeomorphisms of the disk punctured along a Cantor set can be realized as a subgroup of $\Bc_0$ (see \cite{FK}, section 7).  
\end{remark}




\section{The relation with Thompson's groups}

\label{sec:V}
As mentioned in the introduction, the groups $\Bc_g$ and $\Bh_g$ are strongly related to Thompson's groups. This is a manifestation of a more general phenomenon, which we now explain. 

Observe that any homeomorphism of $\Sigma_g$ induces a homeomorphism of the {\em space of ends} (see, for instance, \cite{BH}) of $\Sigma_g$, which by definition is the Cantor set  $C$. Thus we have a 
continuous homomorphism 
\begin{equation}
\epsilon: \Homeo(\Sigma_g) \to \Homeo(C),
\label{eq:homeo1}
\end{equation}
which is surjective when both homeomorphism groups are endowed with their respective compact-open topologies.

Now, every Cantor set on the plane is {\em tame}, meaning that, up to homeomorphism, it is contained in some line; in particular, it is homeomorphic to the standard triadic Cantor set. A theorem of Sc\v{e}pin (see e.g. \cite[Theorem 1]{Verj}) states that any homeomorphism of the standard triadic Cantor set $C\subset [0,1]\subset \R^2$ extends to a homeomorphism  of $\R^2$; moreover, this homeomorphism can be assumed to be the identity outside a large enough ball. In particular, we have: 

\begin{lemma}
The homomorphism $\epsilon: \Homeo(\Sigma_g) \to \Homeo(C)$ is onto. 
\label{lem:cantorsurj}
\end{lemma}

Moreover, $\Homeo(C)$ is a simple group \cite{Anderson}, and thus  the connected component of the identity in $\Homeo(C)$ is trivial. Therefore, $\epsilon$ descends to a continuous surjective homomorphism (using a slight abuse of notation)
\begin{equation}
\epsilon: \Mod(\Sigma_g) \to \Homeo(C),
\label{eq:mod1}
\end{equation}
where $\Mod(\Sigma)$ has been endowed with the quotient topology coming from the compact-open topology on  $\Homeo(\Sigma_g)$. Observe that the kernel of this homeomorphism is exactly the pure  mapping class group $\PMod(\Sigma_g)$, and thus we have a short exact sequence
\begin{equation}
1 \to \PMod(\Sigma_g) \to \Mod(\Sigma_g) \to \Homeo(C) \to 1
\label{eq:ses}
\end{equation}

We are now going to give the version of the exact sequence \eqref{eq:ses} when $\Mod(\Sigma_g)$ is replaced by $\Bc_g$ or $\Bh_g$. As we will see, the subgroup which appears on the right will be (isomorphic to) Thompson's group $V$. We start by giving a definition of this group. 

\subsection{Thompson's group $V$}
\label{subsec:V}
Recall that {\em Thompson's group} $V$ is the group of right-continuous bijections of the unit circle that map the set of dyadic rationals to itself, are differentiable except at finitely many points, and  on every interval of differentiability they are affine maps whose derivatives are powers of 2. The group $V$ is well-known to be finitely presented, with respect to an explicit presentation. We refer the reader to the standard reference \cite{CFP} for a thorough discussion on the different Thompson's groups. 
Extending results of \cite{FK,FK2}, we will prove: 

\begin{proposition}
For every $g\ge 0$, there is a short exact sequence \begin{equation}
1 \to \PMod_c(\Sigma_g) \to \Bc_g \to V \to 1.
\label{eq:sesB}
\end{equation}  
\label{prop:sesB}
\end{proposition}

\begin{remark}
As mentioned in the introduction, Proposition \ref{prop:sesB} is shown in \cite{FK} in the case when $g=0$, where it is also shown that it splits over Thompson's group $T$; compare with Proposition \ref{split}.
\end{remark}

In order to prove Proposition \ref{prop:sesB}, it will be useful to work with a different incarnation of the group $V$ (see \cite{CFP} also). Namely, $V$ is the group of self-transformations of the rooted 3-valent tree $\mathcal{T}$ whose elements are encoded by equivalence classes of triples $(T,T', \sigma)$, where $T$ and $T'$ are finite rooted subtrees of $\mathcal T$ with the same number of leaves, and $\sigma$ is a bijection between the set of leaves of $T$ and $T'$. Such a triple extends to a transformation of $\mathcal{T}$, and the equivalence relation responds to the fact that different triples may extend to the same transformation of $\mathcal{T}$. 

Every element of $V$ induces a homeomorphism of the space of ends of $\mathcal{T}$, which is homeomorphic to the Cantor set $C$.  Thus $V < \Homeo(C)$. Armed with this alternate description, we adapt the arguments of \cite[Section 2]{FK} in order to prove Proposition \ref{prop:sesB}

\begin{proof}[Proof of Proposition \ref{prop:sesB}]
As mentioned above, the case $g=0$ is covered in \cite{FK}, so assume that $g \ge 1$. Fix a rigid structure on $\Sigma_g$, and let 
$c$ be the separating curve used to define it. The dual graph of the pants decomposition $P$ underlying the rigid structure is naturally isomorphic to $\mathcal{T}$, where the root corresponds to the unique pair of pants of $\Sigma_g-P$ having $c$ as boundary curve.

We now define a homomorphism $\Bc_g \to V$ as follows. Let $f\in \Bc_g$ and consider a defining subsurface $Z$ for $f$.  We associate to $f$ the triple $(T,T',\sigma) \in V$, where $T$ (resp. $T'$) is the  subtree of $\mathcal T$ contained in $Z$ (resp. $f(Z)$), and $\sigma$ is the bijection between the sets of leaves of $T$ and $T'$, respectively, induced by the permutation between the sets of boundary components of $Z$ and $f(Z)$ given by $f$.
At this point, one easily checks that this correspondence gives rise to a well-defined surjective homomorphism $\Bc_g \to V$.

We claim that the kernel of this homomorphism is exactly $\PMod_c(\Sigma_g)$. Indeed, suppose $f\in \Bc_g$ maps to the identity in $V$, and so in particular fixes every end of $\Sigma_g$. Consider a  defining subsurface $Z$ for $f$.
 Since $f$ does not permute ends we may assume, up to 
replacing $Z$ by a suitably larger defining subsurface, that every component of $\Sigma_g -Z$ is mapped to itself.
Therefore, after further enlarging $Z$ if necessary, we deduce that $f$ is the identity outside $Z$, and hence has compact support, as desired.
\end{proof}

Next, observe that if we replace $\Mod(\Sigma_g)$ by the asymptotic mapping class group $\Bc_g$ in equation \eqref{eq:ses}, we obtain:
\begin{equation}
1 \to \PMod(\Sigma_g) \cap \Bc_g \to \Bc_g \to \Homeo_{\Bc_g}(C) \to 1,
\label{eq:sesB1}
\end{equation}
where $\Homeo_{\Bc_g}(C)$ denotes the image of $\Bc_g$ in $\Homeo(C)$ under the homomorphism $\epsilon$ of \eqref{eq:mod1}. 
By the same argument that we used to determine the kernel of the homomorphism $\Bc_g \to V$, we have: 

\begin{lemma}
For every genus $g$, we have $\PMod(\Sigma_g) \cap \Bc_g = \PMod_c(\Sigma_g)$.
\label{lem:techcompact}
\end{lemma}

At this point, the combination of equations \eqref{eq:sesB} and \eqref{eq:sesB1} yields: 

\begin{corollary}
 $\Homeo_{\Bc_g}(C)$ is isomorphic to Thompson's group $V$. 
\end{corollary}

In particular, we have deduced that the restriction of the short exact sequence \eqref{eq:SESgen} to $\Bc_g$ is precisely the sequence \eqref{eq:sesB}.

\subsection{A related Higman-Thompson group} We now give a brief description of the subgroup $V_2[\mathbb{Z}_2]$ of $\Homeo(C)$ that appears when restricting the exact sequence \eqref{eq:SESgen} to $\Bh_g$. The elements of $V_2[\mathbb{Z}_2]$ are transformations of $\mathcal{T}$ encoded 
 by equivalence classes of tuples $(T,T',\sigma,\varepsilon)$, where $T, T'$ are subtrees, $\sigma:\partial T\to \partial T'$ is a bijection and $\varepsilon \in (\Z/2\Z)^{\partial T'}$. 
 The group $V_2[\mathbb{Z}_2]$ is an example of the Higman-Thompson groups $V_n[G]$, where $n\in \mathbb{N}$ and $G$ a subgroup of the symmetric group on $n$ elements \cite{FH,BDJ}. More concretely, it is associated to the subgroup $G=\Z_2$ of permutations of $n$ elements 
generated by the involution exchanging $j$ and  $n-j$, for all $j$. (We remark that $V = V_2[{\rm Id}]$ in this context.) 
Observe that we have an obvious inclusion $V<V_2[\mathbb{Z}_2]$. However, by a surprising result of Bleak-Donoven-Jonu\v{s}as \cite{BDJ} we have that, in fact, $V \cong V_2[\mathbb{Z}_2]$, via an explicit element of $\Homeo(C)$ (a {\em cellular automaton}). 

Finally, we stress that $V_2[\mathbb{Z}_2]$ also appeared in \cite{FNe} as the group of 
those homeomorphisms of a Cantor set embedded in $\mathbb{S}^2$ which extend to smooth diffeomorphisms of the sphere.

 Using the same reasoning as in the proof of Proposition \ref{prop:sesB}, we have: 

\begin{proposition}
For every $g\ge 0$, the restriction to the sequence \eqref{eq:SESgen} to $\Bh_g$ yields a short exact sequence \begin{equation}
1 \to \PMod_c(\Sigma_g) \to \Bh_g \to V_2[\mathbb{Z}_2] \to 1.
\label{eq:sesH}
\end{equation}
\end{proposition}





\section{Finite presentability}
\label{sec:proofT1}

In this section we prove Theorem \ref{thm:fg}. As mentioned in the introduction, the case $g=0$ was settled in \cite{FK}, and  will be a key ingredient in our proof:

\begin{theorem}[\cite{FK}]
The group $\Bc_0$ is finitely presented. 
\label{thm:FK}
\end{theorem}

It will be useful for us to give a brief description of the arguments used  in \cite{FK} for proving Theorem \ref{thm:FK}. The first ingredient, which will also play a central r\^ole here, is the following well-known result of Brown \cite{Brown}:

\begin{theorem}[\cite{Brown}]
Let $G$ be a group acting on a connected and simply-connected two-dimensional CW-complex $X$ by permuting its cells. Suppose that:
\begin{enumerate}
\item The stabilizer of each vertex of $X$ is finitely presented;
\item The stabilizer of every edge of $X$ is finitely generated;
\item $X/G$ is compact.
\end{enumerate}
Then $G$ is finitely presented. 
\label{thm:brown}
\end{theorem}

In \cite{FK}, Funar-Kapoudjian applied Theorem \ref{thm:brown} to the action of $\Bc_0$ on  a modification of the so-called {\em pants complex} \cite{HT}  of $\Sigma_0$. More concretely, they first consider a graph $\mathcal P_0$ whose vertices are (isotopy classes of) pants decompositions of $\Sigma_0$ which differ from the pants decomposition underlying the rigid structure in a finite number of curves.
An edge of $\mathcal{P}_0$ is given by two pants decompositions that are related by an {\em elementary move}, meaning that they differ in exactly two curves, which intersect exactly once (resp. twice) if their union fills a one-holed torus (resp. a four-holed sphere); see \cite{HT} for details. 

Using the same proof of the main result of Hatcher-Thurston \cite{HT}, one deduces that the graph $\mathcal{P}_0$ is connected, and  that it becomes a  simply-connected 2-complex after gluing a 2-cell to every {\em triangle}, {\em square}, and {\em pentagon} of $\mathcal{P}_0$. 
However, it turns out that the action of $\Bc_0$ on $\Pc_0$ is not cocompact, as there are infinitely many $\Bc_0$-orbits of squares \cite[Proposition 5.4]{FK}; in particular, Theorem \ref{thm:brown} cannot be applied to this situation. In order to overcome this, Funar-Kapoudjian construct a modification of $\mathcal{P}_0$, called the {\em reduced} pants complex, by considering only two combinatorial types of squares, and show that $\mathcal P_0$ is still simply connected. They then prove that $\Bc_0$ acts on $\mathcal P_0$ satisfying all the hypotheses of Theorem \ref{thm:brown}, as desired. 

In fact, a minor variation of the arguments in \cite{FK} serves to prove the following strengthening of Theorem \ref{thm:FK}, which will be crucial for us:

\begin{proposition}
For every $b\in \N\cup \{0\}$, the groups $\Bc_0^b$ and $\Bh_0^b$ are finitely presented.  
\label{prop:extFK2}
\end{proposition}
\begin{proof}[Sketch proof]
The definition of the reduced pants complex  from \cite{FK} makes sense also for $\Sigma_0^b$, with arbitrary $b$. By the same arguments as in \cite{HT}, this complex is connected and simply-connected and, using the same reasoning as in \cite{FK}, the groups $\Bc_0^b$ and $\Bh_0^b$ each act cocompactly on it. 
The stabilizer of every cell is  an extension of a finite permutation group by a finitely generated 
free-abelian group, and in particular finitely presented. At this point, the result follows from Theorem \ref{thm:brown}. 
\end{proof}
 
We now turn to the proof of Theorem \ref{thm:fg}, whose statement we now recall: 

\begin{named}{Theorem \ref{thm:fg}}
The groups  $\Bc_g$ and $\Bh_g$ are finitely presented for every $g\ge 0$. 
\end{named}

In light of Theorem \ref{thm:FK}, it suffices to show the result for $g$ positive. In order to do so, we will also use Brown's Theorem \ref{thm:brown}, this time using the action of $\Bc_g$ (resp. $\Bh_g$) on the so-called {\em cut-system complex} $\mathcal{K}_g$ of $\Sigma_g$. We remark that cut-system complexes of finite-type surfaces were in fact used by Hatcher-Thurston \cite{HT} and Wajnryb \cite{Waj} in order to compute finite presentations of mapping class groups. 

We now introduce the complex $\mathcal{K}_g$. For concreteness, we choose to write the definitions for the surface $\Sigma_g$, although they make sense for an arbitrary connected orientable surface, of finite or infinite type. 

   A {\em cut system} of $\Sigma_g$ is a set of $g$ non-separating curves  whose union does not separate $\Sigma_g$. Let $\mathcal{K}_g$ be the simplicial graph whose vertices are cut systems on $\Sigma_g$, and where two cut systems are adjacent in $\mathcal{K}_g$ if they differ in exactly two curves, which intersect exactly once. 

Moreover, similar to the case of the pants complex, $\mathcal{K}_g$ will become a simply-connected 2-complex after gluing a 2-cell to certain circuits in $\mathcal{K}_g$. Before explaining this, we need to borrow  some definitions from \cite{HT,Waj}.  A {\em triangle} in $\mathcal{K}_g$ consists of three pairwise-adjacent vertices of $\mathcal{K}_g$; geometrically, a triangle corresponds to three  curves that pairwise intersect exactly once.  A {\em square} in $\mathcal{K}_g$ is a closed path with four vertices $v_1, \ldots,v_4$ such that $v_i$ and $v_j$ are adjacent in $\mathcal{K}_g$ if and only if $|i-j| =1 \mod 4$; geometrically, a square corresponds to two elementary moves that occur in two disjoint one-holed tori. Finally, a {\em pentagon} in $\mathcal{K}_g$ consists of five vertices  $v_1, \ldots,v_5$ such that $v_i$ and $v_j$ are adjacent in $\mathcal{K}_g$ if and only if $|i-j| =1 \mod 5$; geometrically, a pentagon is determined by five curves $c_1, \ldots, c_5$ on $\Sigma_g$ such that both $c_i$ and $c_{i+1}$ belong to the cut system  $v_i$ and $i(c_i,c_{i+2}) =1$, counting indices modulo 5. 

Armed with these definitions, we construct a 2-complex by gluing a 2-cell to every triangle, square, and pentagon of $\mathcal{K}_g$. By a slight abuse of notation, we will denote the resulting complex by $\mathcal{K}_g$ also. The following result is essentially due to Hatcher-Thurston \cite{HT} and Wajnryb \cite{Waj}:

\begin{theorem}[\cite{HT,Waj}]
For every $g\ge 1$, the complex $\mathcal{K}_g$ is connected and simply-connected. 
\label{thm:HT}
\end{theorem}

\begin{proof}
First, Hatcher-Thurston \cite{HT} and Wajnryb \cite{Waj} proved that the cut-system complex of a finite-type surface is connected and simply-connected. To see that this is also the case for $\mathcal{K}_g$ observe that, for every finite set of vertices $A \subset \mathcal{K}_g$, the union of the curves defining the elements of $A$ together fill a finite-type subsurface of $\Sigma_g$. 
\end{proof}

We are finally in a position to prove Theorem \ref{thm:fg}: 

\begin{proof}[Proof of Theorem \ref{thm:fg}] We prove the result for $\Bc_g$; the same argument, replacing every instance of $\Bc$ by $\Bh$, will give the result for $\Bh_g$. 
As mentioned earlier, we are going to apply Theorem \ref{thm:brown} to the action of $\Bc_g$ on $\mathcal{K}_g$. 

First of all, Theorem \ref{thm:HT} tells us that $\mathcal{K}_g$ is connected and simply-connected. Now, the classification theorem for infinite-type surfaces \cite{Richards} and the fact that vertices are defined by a finite set of curves, together imply that $\Bc_g$ acts transitively on the set of vertices of $\mathcal{K}_g$. Since edges and 2-cells of $\mathcal{K}_g$ are defined in terms of intersection numbers, and by a finite set of curves, we deduce that $\Bc_g$ acts cocompactly on $\mathcal{K}_g$. 

Thus, it remains to justify why the stabiliser of a vertex (resp. edge) is finitely presented (resp. finitely generated). Consider first the case of the stabiliser of a vertex $u$  of $\mathcal K_g$. Fix an orientation on each curve of $u$. Up to passing to a subgroup of finite index, and abusing notation, we may assume that every element of $\Stab(u)$ fixes every curve of $u$ together with its orientation. Now, cutting $\Sigma_g$ open along the elements of $u$ we obtain a surface homeomorphic to $\Sigma_0^{2g}$, and a short exact sequence \[1\to \Z^g \to \Stab(u) \to \Bc_0^{2g}\to 1,\] where $\Z^g$ is the group generated by the Dehn twists along the elements of $u$. Now, an extension of a finitely presented group by a finitely presented group is also 
finitely presented and thus Proposition  \ref{prop:extFK2} yields that (a finite-index subgroup of) $\Stab(u)$ is finitely presented, as desired. 

Observe that the surface obtained from $\Sigma_g$ by cutting the surface along the curves defining an edge of $\mathcal{K}_g$ is homeomorphic to $\Sigma_0^{2g-1}$. In light of this, as above we deduce that (a finite-index subgroup of) the stabiliser of an edge $e\in \mathcal{K}_g$ fits in a short exact sequence 
\[1\to \Z^{g-1} \to \Stab(e) \to \Bc^{2g-1}_0 \to 1.\]  
Therefore, the stabiliser of an edge is finitely presented, and in particular finitely generated. 
This finishes the proof of Theorem \ref{thm:fg}.
\end{proof} 

 
\section{The group of half-twists is dense}
\label{sec:dense}
The goal of this final section is to prove Theorem \ref{thm:limit}, whose statement we now recall: 

\begin{named}{Theorem \ref{thm:limit}}
 $\Bh_g$ is dense in $\Mod(\Sigma_g)$. 
\end{named}

\begin{remark} In light of the result of Neretin and the second author \cite{FNe} mentioned above, when $g=0$ this result may be interpreted as stating that homeomorphisms of the sphere minus a Cantor set may be approximated by diffeomorphisms. 
\end{remark}

As it turns out, Theorem \ref{thm:limit} will be a consequence of a slightly stronger result, namely Theorem \ref{thm:limit2} below. Before we state it, we need some preliminaries. Fix, once and for all, a rigid structure on $\Sigma_g$. Recall from section \ref{sec:rigid} that the rigid structure is given by a separating curve on $\Sigma_g$, which will be denoted  $c(\emptyset)$  for reasons that will become apparent below, plus a rigid structure on the planar component $\Sigma_g^* $ of $\Sigma_g - c(\emptyset)$. Let $P$ be the pants decomposition underlying the rigid structure on $\Sigma_g^* $. 

Similar to the situation  in subsection \ref{subsec:V}, there is an infinite rooted tree $\mathcal T'$ associated to $P$, whose vertices are the curves of $P$, with the root being $c(\emptyset)$. As such, every vertex of $\mathcal T'$  is naturally labelled by a word $w$ in the free semigroup $F(L,R)$ generated by the two letters $L$ (left) and $R$ (right). We denote by $c(w)$ the curve of $P$ labelled by the word $w\in F(L,R)$, and write $h(w)$ for the half-twist about $c(w)$.

In turn, this labelling induces a labelling of the set of pairs of pants of $\Sigma_g^* $ by words in $F(L,R)$. Indeed, we set $P(w)$ to be the unique $P$-suited pair of pants of $\Sigma_g^* $ that has $c(w)$ as boundary component and is contained in the planar component of $\Sigma_g  - c(w)$.

Finally, set $S(-\infty)$ to be a fixed compact genus-$g$ subsurface of $\Sigma_g$ which contains $c(\emptyset)$. In addition, for each $w \in F(L,R)$, let $S(w)$ be a $P$-suited (therefore compact) subsurface whose interior is contained in $\Sigma_g^* - c(w)$, and which has $c(w)$ as a boundary component. Observe that the choice is far from unique; however, for any such choice, the set \[\{S(w) \mid w\in F(L,R) \cup \{-\infty\} \}\] is {\em proper}, in the sense that every compact subsurface of $\Sigma_g$ intersects only finitely  many elements of this set. Because of this, we will refer to any set $\{S(w) \mid w\in F(L,R) \cup \{-\infty\}\}$ as above as a 
{\em proper exhaustion by subsurfaces}. 

After all this discussion, we have the following definition: 

\begin{definition}[Proper sequence of mapping classes]
A {\em proper sequence of mapping classes}  is a sequence $\{f_n\}_{n\in \N} \subset \Mod(\Sigma_g)$, where 
\[f_n = \prod_{\begin{subarray}{c} w\in F(L,R) \cup \{-\infty\} \\ |w| \le n \end{subarray}} f(S(w)) h(w)^{\epsilon_w},\] where: 
\begin{enumerate}
\item $ f(S(w))$ is an element of $\PMod(S(w))$,
\item $|w|$ denotes the length of the word $w$,
\item $\epsilon_w \in\{-1,0,1\}$ for all $w$,
\item The product is ordered lexicographically, and is defined from left to right.
\end{enumerate}
\end{definition}
Observe that if $\{f_n\}_{n\in \N} \subset \Mod(\Sigma_g)$ is a proper sequence, then $\{f_n\}_{n\in \N} \subset \Bh_g$ as well. In addition, since the defining subsurfaces of the $f_n$ form a proper exhaustion of subsurfaces of $\Sigma_g$, we have:

\begin{lemma}
If $\{f_n\}_{n\in \N} \subset \Bh_g$ is a proper sequence, then it has a limit in $\Mod(\Sigma_g)$. 
\label{lem:proper}
\end{lemma}

 Recall that $\Mod(\Sigma_g)$ his equipped with the quotient topology coming from the compact-open topology on $\Homeo(\Sigma_g)$. We will prove: 

\begin{theorem}
Every element of $\Mod(\Sigma_g)$ is a limit of a proper sequence in $\Bh_g$. 
\label{thm:limit2}
\end{theorem}

Observe that  Theorem \ref{thm:limit2} obviously implies Theorem \ref{thm:limit}.

\begin{proof}[Proof of Theorem \ref{thm:limit2}]
Let $f \in \Mod(\Sigma_g)$ be an arbitrary mapping class. We will use an inductive argument, which we call a {\em straightening} of $f$, to produce a proper sequence of mapping classes whose limit is $f$. 

Consider $c(\emptyset)$. By the classification of infinite-type surfaces  \cite{Richards}, there exist a $P$-suited genus-$g$ subsurface $S(-\infty)$, and an element $f_{{\tiny{-\infty}}}$ such that $f_{{\tiny{-\infty}}}  f$ sends $c(\emptyset)$ to itself. We say that $f_{{\tiny{-\infty}}}  f$  {\em straightens}  $c(\emptyset)$.

Next, consider the curves $c(L)$ and $c(R)$. Again, there exists a $P$-suited planar subsurface $S(\emptyset)$ and an element $f_\emptyset\in\PMod(S(\emptyset))$ such that $f_\emptyset f_{-\infty}  f$ sends the set $\{c(L),c(R)\}$ to itself. Therefore, up to precomposing with the half-twist $h(\emptyset)$ about $c(\emptyset)$ if necessary, we may in fact assume that $c(L)$ (resp. $c(R)$) is sent to itself. We say that $f_\emptyset f_{-\infty}  f$ straightens the pair of pants $P(\emptyset)$.

We continue this process inductively to find, for all $n$, a proper exhaustion by (compact) $P$-suited subsurfaces $S(w)$ and elements $f(S(w)) \in \PMod(S(w))$ such that, setting \[f_n:=  \prod_{\begin{subarray}{c} w\in F(L,R) \cup \{-\infty\} \\ |w| \le n \end{subarray}} f(S(w)) h(w)^{\epsilon_w}, \] the mapping class 
\[(f_n f)_{n\in \N}\]  straightens every pair of pants $P(w)$ for $|w|\le n$. 

By Lemma \ref{lem:proper}, the sequence $(f_n)_{n\in \N}$ converges to an element $f^*\in \Mod(\Sigma_g)$, in such way that $f^* f$ fixes $c(\emptyset)$ and every curve of $P$. Moreover, by composing with an appropriate power of a Dehn twist at each step, we may assume that $f^* f$ induces the trivial twist on each curve of $P$. Therefore, $f^* f$ is an element of $\PMod(S({-\infty}))$. In other words, we have shown that, up to precomposing $f_\emptyset$ with the inverse of this element, the sequence $f_n$ converges to $f$ in $\Mod(\Sigma_g)$, as desired.
\end{proof}




\section{Stable homology} 
\label{sec:homology}
In this section, we adapt the methods of \cite{FK2} to prove Theorem \ref{thm:homology}.
 First, observe that for $n \ge 1$, gluing a pair of pants to a boundary component of $S_{g,n}$ gives rise to an injective homomorphism \[\PMod(S_{g,n}) \to \PMod(S_{g,n+1}).\]  In particular, one has a homomorphism 
\begin{equation}
H_k(\PMod(S_{g,n}),\Z) \to H_k(\PMod(S_{g,n+1}),\Z)
\label{eq:homology}
\end{equation}
between the corresponding homology groups, which has a one-sided inverse coming from the {\em capping homomorphism} \cite{FM} and is therefore injective.   In \cite{Harer}, Harer proved that the $k$-th homology group of $\PMod(S_{g,n})$ does not depend on $n$, provided $k$ is sufficiently small with respect to $g$. The current best bound for what ``sufficiently small" means is due to Boldsen \cite{Boldsen} who, building up on unpublished work of Harer \cite{Harer-unp}, proved the following: : 
 
\begin{theorem}[\cite{Harer,Boldsen}] Let $g,n,k  \in \Z$, with $g\ge 0$ and $n, k\ge 1$. Then, the homomorphism \eqref{eq:homology} is an isomorphism for every $k\le 2g/3$.
\label{thm:harer}
\end{theorem}
  
At this point, Lemma \ref{lem:limit} and Theorem \ref{thm:harer} together imply that 
\[H_k(\PMod_c(\Sigma_g), \mathbb{Z}) \cong H_k(\PMod(S_{g,n}),\mathbb{Z})\] 
for all $n\ge 1$,  provided $k \le 2g/3$. For this reason, we will refer to the homology group $H_k(\Mod_c(\Sigma_g), \mathbb{Z}) $ as the {\em stable}  $k$-th homology group of the mapping class group of genus $g$. 

We are finally in a position to give a proof of Theorem \ref{thm:homology}, whose statement we now recall:

\begin{named}{Theorem \ref{thm:homology}}
Let $g\ge 1$. For every $k \le 2g/3$, $H_k(\Bc_{g}, \Z)$ (resp. $H_k(\Bh_{g}, \Z)$) is isomorphic to the $k$-th stable homology group of the mapping class group of genus $g$. 
\end{named}
 
\begin{proof} 
We adapt the proof of \cite[Theorem 3.1]{FK2} to our setting. We treat the case of $\Bc_g$ only, the other one being a direct translation. 
It is well-known that $\PMod(S_{g,n})$, where $n\geq 1$, is  torsion free and has a finite dimensional classifying space, which can be taken to be manifold (see e.g. \cite{Bo,Godin}). 
Thus their integral homology is of finite type and by Harer's Stability Theorem 
 $H_k(\PMod_c(\Sigma_g), \Z)$ is 
finitely generated for every $k$. 

Now, we apply the Lyndon-Hochschild-Serre spectral sequence method to
the short exact sequence \eqref{eq:sesB}, which reads 
\[1 \to \PMod_c(\Sigma_g) \to \Bc_g \to V \to 1.\] 
In other words, there is a spectral sequence approximating the homology of $\Bc_g$ whose second page is 
\[E^2_{pq}= H_p(V,H_q(\PMod_c(\Sigma_g),\Z)).\] 
We now claim  that $V$ acts trivially on $H_q(\PMod_c(\Sigma_g), \Z)$ for all $q$.
Accepting this claim for the moment, it follows that the only non-zero 
terms of the spectral sequence above for the rational homology are  those with $p=0$, because $V$ is acyclic (see \cite{SW,Brown2}). 
In particular, the spectral sequence for the homology of $\Bc_g$ collapses at the second page, which implies that 
\[H_q(\Bc_g, \Z)\cong H_q(\PMod_c(\Sigma_g),\Z)\] for all $q$, as desired. Thus it remains to show:

\medskip

\noindent{\bf{Claim.}} $V$ acts trivially on $H_q(\PMod_c(\Sigma_g), \Z)$. 

\begin{proof}[Proof of claim]
Since $H_k(\PMod_c(\Sigma_g), \Z) $ is a finitely generated abelian group, we get a finite-dimensional representation $V \to \GL(N, \mathbb Q)$. As $V$ is simple, it follows that this representation is either trivial or injective. On the other hand, every finitely-generated subgroup of $\GL(N, \mathbb{Q})$ is residually finite, while $V$ is not. In particular, the given representation is trivial.
\end{proof}
This finishes the proof of Theorem \ref{thm:homology}.
\end{proof}

\begin{proof}[Proof of Corollary \ref{cor:perfect}]
By the sequence \eqref{eq:sesB}, plus the fact that $\PMod_c(\Sigma_g)$ ($g\geq 3$)
and $V$ are both perfect, we obtain that $\Bc_g$ and $\Bh_g$ are perfect.  
\end{proof}



\section{Rigidity}
\label{sec:proofaut}

The goal of this section is to prove Theorem \ref{thm:autos}. 
As mentioned above $\Aut(G)$ will denote the automorphism group of the group $G$. 

\subsection{Automorphisms of $\PMod_c(\Sigma_g)$}
The first ingredient in the proof of Theorem \ref{thm:autos} is the following: 

\begin{proposition}
Suppose $g\ge 0$. Then $\Aut(\PMod_c(\Sigma_g))= \Mod(\Sigma_g)$. 
\label{prop:autos_pmod}
\end{proposition}

\begin{remark}
We note that, prior to this work, this result had been obtained by Patel-Vlamis \cite{PV} for $g\ge 4$. 
\end{remark}

Before explaining a proof of Proposition \ref{prop:autos_pmod}, we note the following easy observation:

\begin{lemma}
$\PMod_c(\Sigma_g)$ is a normal subgroup of $\Mod(\Sigma_g)$ for every $g\ge 0$.
\label{lem:pmod_normal}
\end{lemma}

\begin{proof}
Let $g\in \PMod_c(\Sigma_g)$  and let $Z$ be a  support for $g$. Consider an arbitrary $h \in \Mod(\Sigma_g)$, noting that any $P$-suited subsurface $W$ containing $h(Z)$ is a support for $hgh^{-1}$. Moreover, $hgh^{-1}$ induces the trivial permutation on the set of boundary components of $W$, and therefore $hgh^{-1}\in \PMod_c(\Sigma)$.
\end{proof}

In order to prove Proposition \ref{prop:autos_pmod}, we follow Ivanov's strategy for proving that automorphisms of (sufficiently complicated) finite-type mapping class groups are  conjugations. First, we observe:

\begin{lemma}
Every automorphism of $\PMod_c(\Sigma_g)$ sends  Dehn twists to Dehn twists. 
\label{lem:twists2twists}
\end{lemma} 

We stress that, if $4\le g <\infty$, then Lemma \ref{lem:twists2twists} quickly follows from \cite[Corollary 1.5]{AS} plus the fact that $\PMod_c(\Sigma_g)$ is a direct limit of $P$-suited subsurfaces of $\Sigma_g$ (Lemma \ref{lem:limit}). We now present a simpler argument, valid for arbitrary $g$, which was suggested to us by Juan Souto: 

\begin{proof}
Let $\phi \in \Aut(\PMod_c(\Sigma_g))$. We take an arbitrary $P$-suited subsurface $Z\subset \Sigma_g$ of genus $g$; in addition, we will assume that $Z$ has at least seven boundary components (note that the latter assumption is only relevant if $g<3$). Since $\Mod(Z)$ is finitely presented and $\PMod_c(\Sigma_g)$ is a direct limit of pure mapping class groups of $P$-suited subsurfaces of genus $g$, by Lemma \ref{lem:limit}, we deduce that there exists a $P$-suited subsurface $W\subset \Sigma_g$ such that $\phi(\Mod(Z)) < \Mod(W)$. Let $W^*$ be the (compact) subsurface of $W$ supporting $\phi(\Mod(Z))$; in particular, this implies  that $\phi(\Mod(Z)) < \Mod(W^*)$, and that $\Mod(W^*)$ and $\phi(\Mod(Z))$ have the same centralizer in $\PMod_c(\Sigma_g)$. 

Now, observe that, given a compact subsurface $Y\subset \Sigma_g$, the rank of the center of the centralizer of $\Mod(Y)$ in $\PMod_c(\Sigma_g)$ is equal to the number of boundary components of $Y$. In particular, since $\Mod(Z)$ and $\phi(\Mod(Z))$ are isomorphic, the discussion of the previous paragraph implies that $W^*$ and $Z$ have the same number of boundary components. Again, \cite{AS} or \cite{Castel} imply that they have equal genus, as $Z$ has genus $g$ and $\phi(\Mod(Z)) < \Mod(W^*)$. In other words,  $Z$ and $W^*$ are homeomorphic. We claim:

\medskip

\noindent{\em Claim.} The isomorphism $\phi: \Mod(Z) \to \Mod(W^*)$ is induced by a homeomorphism $Z\to W^*$. 

\medskip

\begin{proof}[Proof of Claim]
If $Z$ has genus $\ge  2$, then the claim is \cite[Theorem 9]{Castel}. We now sketch an argument valid for arbitrary genus. 

The center $C_Z$ of $\Mod(Z)$ is the free abelian group generated by the Dehn twists along the boundary components of $Z$. Observe that $\phi$ induces an isomorphism 
\begin{equation}
\Mod(Z) /C_Z \to \Mod(W^*)/C_{W^*},
\label{eq:homo}
\end{equation} and that $\Mod(Z) /C_Z$ is isomorphic to the pure mapping class group of the surface which results from gluing a punctured disk to every boundary component of $Z$. By the work of numerous authors (see, for instance, \cite[Theorem 1.3]{BeMa} for a comprehensive statement), the isomorphism  \eqref{eq:homo} is induced by a homeomorphism $f$ between the punctured surfaces. 
To see that $f$ extends to a homeomorphism $Z\to W^*$ inducing $\phi$, we use the well-known {\em lantern relation} (see, for instance, \cite[Section 5]{FM}). Indeed, since we are assuming that $Z$ has at least seven boundary components, given a boundary component $a \subset \partial Z$ we may find a collection $b_1,\ldots, b_6$ of six curves, all of them essential in $Z$, such that the Dehn twist along $a$ is expressed as a product of (suitable powers of) Dehn twists along the $b_i$'s, via the lantern relation. 
Thus the claim follows. 
\end{proof}

The claim above  implies that $\phi: \Mod(Z) \to \Mod(W^*)$ is induced by a homeomorphism $Z \to W^*$, and in particular sends Dehn twists to Dehn twists. 
The lemma now follows since $Z$ was arbitrary and every curve on $\Sigma_g$ is contained in some $P$-suited subsurface of genus $g$. 
\end{proof}

Continuing with the arguments towards a proof of Proposition \ref{prop:autos_pmod}, we next claim that every automorphism of $\PMod_c(\Sigma_g)$ induces an automorphism of the  
curve complex $\mathcal{C}(\Sigma_g)$; we refer the reader to the articles  \cite{HMV,BDR} for various features of these complexes in the infinite-type setting.

Indeed, let
 $\phi : \PMod_c(\Sigma_g) \to \PMod_c(\Sigma_g)$ be an automorphism. By Lemma \ref{lem:twists2twists}, given a curve $c \in \Sigma_g$ there exists a curve $d$ such that $\phi(T_c) = T_d$, where $T_a$ denotes the (left) Dehn twist about the curve $a$. Since $T_a = T_b$ if and only if $a = b$, we get that the curve $d$ above is in fact unique, and hence  $\phi$ induces a well-defined map $$\phi_* : \Cc(\Sigma_g) \to \Cc(\Sigma_g)$$ by the rule $\phi_*(c) = d$.
 Observe that the map $\phi_*$ is simplicial, since two Dehn twists commute if and only if the defining curves are disjoint. Moreover, it is bijective, with inverse the simplicial self-map of $\Cc(\Sigma_g)$ induced by $\phi^{-1}$. 

Now we need the following analogue of Ivanov's celebrated result \cite{Ivanov} on automorphisms of the curve complex:

\begin{theorem}[\cite{HMV, BDR}]\label{autcomplex}
 The automorphism group of $\Cc(\Sigma_g)$ is isomorphic to $\Mod(\Sigma_g).$
\label{thm:hmv}
\end{theorem}

We are finally in a position to prove Proposition \ref{prop:autos_pmod}: 

\begin{proof}[Proof of Proposition \ref{prop:autos_pmod}]
We wish to show that the natural homomorphism \[\Lambda: \Mod(\Sigma_g) \to \Aut(\PMod_c(\Sigma_g)) \] is an isomorphism. We first prove that $\Lambda$ is surjective. To this end, let 
 $\phi\in \Aut(\PMod_c(\Sigma_g))$, and consider the automorphism $\phi_* : \Cc(\Sigma_g) \to \Cc(\Sigma_g)$ induced by $\phi$. By Theorem \ref{thm:hmv}, there exists a unique $f \in \Mod(\Sigma_g)$ with $f(c) = \phi_*(c)$ for all $c \in \Cc(\Sigma_g)$. 
 Note that $\psi:=f^{-1}\phi$ is an element of $\Aut(\PMod_c(\Sigma_g))$, and that  
 $\psi_*(c) =c$ for all $c \in \Cc(\Sigma_g)$. 
 
Using again a direct translation of a classical argument in the finite-type setting, we will prove that the automorphism $\psi$ is trivial, so that $\phi$ is in fact the conjugation by $f$. 
To this end, let $h\in \Mod(\Sigma_g)$, and choose an arbitrary curve $c$ on $\Sigma_g$. We have: 
\[\psi(hT_c h^{-1}) = \psi(T_{h(c)}) = T_{\psi_*(h(c))}= T_{h(c)}.\]
On the other hand: 
\[\psi(hT_c h^{-1}) =\psi(h) \psi(T_c) \psi(h)^{-1} =  \psi(h) T_{\psi_*(c)} \psi(h)^{-1}= T_{\psi(h)(c)}\]
Combining both equations, we obtain that $h(c) = \psi(h)(c)$ for every curve $c$ on $\Sigma_g$.   
By Theorem \ref{autcomplex} we deduce that $h = \psi(h)$ for every $h\in \PMod_c(\Sigma_g)$. In other words,  $\psi$ is the identity on $\PMod_c(\Sigma_g)$, as desired. 
\end{proof}

\subsection{Proof of Theorem \ref{thm:autos}}
We now embark in the proof of Theorem \ref{thm:autos}.  The key ingredient is the following unpublished result of Kapoudjian: 

\begin{lemma}[Kapoudjian]\label{lemma-kapou}
Let  $f \in \Bc_g$. Then $f\in \PMod_c(\Sigma_g)$ if and only if for any finite set $\{f_i\}\subset \Bc_g$ there exists $h\in \Bc_g$ such that $[h,f_iff_i^{-1}]=1$ and
$f$ belongs to  the normal closure of $h$ in $\Bc_g$. 
\end{lemma}

\begin{proof}
Let $f\in \Bc_g$. Suppose first that $ f\in \PMod_c(\Sigma_g)$, and let $\{f_i\}\subset \Bc_g$ be an arbitrary finite set. Choose a minimal  support $Z$ for $f$, so that $f(Z)=Z$ and $f$ is identity outside $Z$.  
Observe that $f_i(Z)$ is a  support for $f_iff_i^{-1}$. Then we may take $h_0\in \Bc$ such that $h_0(Z)\cap f_i(Z)=\emptyset$ for all $i$, and set 
$h=h_0fh_0^{-1}$, which commutes with $f_iff_i^{-1}$. 

For the opposite direction let $Z$ be a defining surface for $f$; recall this means that $f$ sends $Z$ to $f(Z)$, and is rigid in the complement of $Z \cup f(Z)$. Up to enlarging $Z$ if necessary, we may assume that $Z$ has at
 least five boundary components. We now choose some finite set $\{f_i\}\subset \Bc_g$ as in the statement, and which furthermore satisfies that the $f_i$ all send $Z$ to itself, and that the group generated by the $f_i$ induces every possible permutation of the set of boundary components of $Z$.

Seeking a contradiction, suppose that $f\notin \PMod_c(\Sigma_g)$. Then there exists a connected component of $\Sigma_g-Z$ whose image by $f$ is disjoint from itself. 
This implies that every defining surface of any $h \in \Bc_g$ as in the statement  must be contained in $Z$. Furthermore, the permutation of the set of boundary components of $Z$ induced by $h$ must commute with every element of the conjugacy class of the (nontrivial) permutation induced by $f$. 

Now, we claim that if a permutation $\sigma$ in the symmetric group on $k\ge 5$ elements 
 commutes with every conjugate of a fixed nontrivial permutation $\sigma_0$, then $\sigma$ is the trivial permutation. Indeed, $\sigma$ commutes then with every element of the normal subgroup generated 
by $\sigma_0$; as $k\geq 5$ this normal subgroup contains the alternating group.  
Further, $\sigma$ cannot be even as the alternating group is center-free. 
It follows that $\sigma$ commutes with the even permutation $\tau\sigma$ and hence $\sigma$ commutes with $\tau$, for any transposition $\tau$. Then $\sigma$ is in the center of the symmetric group which is trivial.

In particular,  the permutation induced by $h$ is the trivial one, and $h \in \PMod_c(\Sigma_g)$. Hence $f$ belongs to the normalizer of $ \PMod_c(\Sigma_g)$ in $\Bc_g$, contradicting our assumption. 
\end{proof}

As an immediate consequence we get: 

\begin{corollary}
$\PMod_c(\Sigma_g)$ is a characteristic subgroup of $\Bc_g$.
\label{cor:charac}
\end{corollary}

\begin{remark}\label{H-kapou}
The statement of Lemma \ref{lem:pmod_normal} and the proof are still valid 
when we replace $\Bc_g$ by $\Bh_g$. 
This implies that $\PMod_c(\Sigma_g)$ is  characteristic in $\Bh_g$ also.
\end{remark} 

In  light of Corollary \ref{cor:charac}, every automorphism of $\Bc_g$ (resp. $\Bh_g$) induces an automorphism of $\PMod_c(\Sigma_g)$, which we have already determined in Proposition \ref{prop:autos_pmod}. Now, in order to calculate the automorphism group of $\Bc_g$ and $\Bh_g$ from this, we will make use of the following small technical result -- this is surely well-known, but we include a proof for completeness. 

\begin{lemma}
Let $H$ be a normal subgroup of a group $G$, and suppose it has trivial centralizer in $G$. Suppose $\psi: G\to G$ is an injective homomorphism of $G$ such that $\psi_{| H}= {\rm id}_{H}$. Then $\psi=  {\rm id}_{G}$. 
\label{lem:restriction}
\end{lemma}

\begin{proof}
Since $\psi$ is the identity on $H$ and $H$ is normal in $G$ we have that, for every $f\in G$ and every $h \in H$, $$fhf^{-1}= \psi(fhf^{-1}) = \psi(f) \psi(h) \psi(f^{-1}) =  \psi(f) h \psi(f^{-1}).$$ Since $h$ is arbitrary, it follows that $\psi(f^{-1})f$ belongs to the centralizer of $H$ in $G$, and therefore $\psi(f) = f$, by hypothesis. Since $f$ is also arbitrary, the result follows.
\end{proof}

The motivation for Lemma \ref{lem:restriction} is the following claim: 

\begin{lemma}
The centralizer of $\PMod_c(\Sigma_g)$ in $\Mod(\Sigma_g)$ is trivial. 
\label{lem:centralizer}
\end{lemma}

\begin{proof} 
Suppose $f\in \Mod(\Sigma_g)$ commutes with every element of $\PMod_c(\Sigma_g)$, and in particular with every Dehn twist. In particular, we have that 
\[T_a= fT_a f^{-1} = T_{f(a)},\] 
for every curve $a$ on $\Sigma_g$. But, as we mentioned in the proof of Proposition \ref{prop:autos_pmod}, the Alexander method of \cite{HMV} implies that every element of $\Mod(\Sigma_g)$ that fixes every curve on $\Sigma_g$ is the identity.
\end{proof}

We need the following definition before proving Theorem \ref{thm:autos}. Suppose $Z$ is an orientable surface with non-empty boundary $\partial Z$, where we assume that every connected component of $\partial Z$ has been parametrized by means of a map $\varphi:\partial Z\to S^1$, which is a homeomorphism on each component. 
A homeomorphism $F:Z\to Z$ respects the boundary parametrization if  
$\varphi\circ f|_{\partial Z}=\varphi$. 
The {\em parametrized mapping class group} $\Mod^*(Z)$ is the 
group of isotopy classes of self-homeomorphisms of $Z$ respecting the boundary parametrization, modulo 
isotopies which are the identity on the boundary. 
Observe that $\Mod(Z) < \Mod^*(Z)$ is a finite index normal subgroup and that, as a by-product of the definition, an element of $\Mod^*(Z)$ may induce a non-trivial permutation on the set of boundary components of $Z$.

\begin{proof}[Proof of Theorem \ref{thm:autos}]  We prove the result for $\Bc_g$, as the case of $\Bh_g$ is dealt with along similar lines.
We first show that the natural homomorphism $$N_{\Mod(\Sigma_g)}(\Bc_g) \to \Aut(\Bc_g)$$ is surjective. To this end, let
 $\phi: \Bc_g \to \Bc_g$ be an automorphism. By Corollary \ref{cor:charac}, $\phi$ induces an automorphism  $$\phi_c: \PMod_c(\Sigma_g) \to \PMod_c(\Sigma_g),$$ which  is the restriction of $\phi$ to $\PMod_c(\Sigma_g)$.  Proposition \ref{prop:autos_pmod} implies that there exists $\Phi\in \Mod(\Sigma_g)$ such that 
 $\phi_c$ is conjugation by $\Phi$, denoted    $Ad_{\Phi}$. 
 
 From the proof of Lemma \ref{lem:twists2twists} for any (compact) $P$-suited subsurface $X\subset \Sigma_g$ of genus $g$, there  exists a compact subsurface $Y=\Phi(X)\subset \Sigma_g$ such that the restriction  $\phi_c|_{\Mod(X)}$ sends  $\Mod(X)$ isomorphically onto $\Mod(Y)$. 
 
Observe that for any $P$-suited subsurface $X\subset \Sigma_g$ the group $\Mod^*(X)$ admits a canonical embedding into 
$\Bc_g$, by extending rigidly classes of homeomorphisms of $X$ to classes of homeomorphisms of $\Sigma_g$. We claim:

%
%
%
 \begin{lemma}
 For any $P$-suited  subsurface $X\subset \Sigma_g$ of genus $g$, the restriction 
 $\phi|_{\Mod^*(X)}$ sends  $\Mod^*(X) $ isomorphically onto  some copy of $\Mod^*(\Phi(X))$.
 Moreover,  $\phi|_{\Mod^*(X)}$ is conjugation by $\Phi$. 
 \label{lem:Ad}
 \end{lemma}
 \begin{proof}[Proof of the Lemma]
 As $\phi|_{\Mod(X)}$ is conjugation  by $\Phi$, we have 
 $\phi(T_c)=T_{\Phi(c)}$, when $c\in \partial X$ is a boundary circle. 
 The stabilizer of the set $\{T_c \mid c\in \partial X\}$ in $\Bc_g$ is the stabilizer of the simplex of $\mathcal C(\Sigma_g)$ 
 determined by elements of $\partial X$, with respect to the $\Bc_g$-action.   
 We call it the stabilizer of the multicurve. 
 
Let $Y$ be a compact subsurface of $\Sigma_g$ and denote by $S_{\partial Y}$ the permutation group on the set of boundary components of $Y$. Looking at the induced permutation on the set of boundary components, we obtain canonical homomorphisms: 
\[ \Mod^*(Y)\to S_{\partial Y}, \hspace{.5cm} \Mod^*(\Sigma_g \setminus \inte(Y))\to S_{\partial Y},\]
where $\inte(Y)$ denotes the interior of $Y$.  We denote by $\Mod^*(Y)\times _{S_{\partial Y}} \Mod^*(\Sigma_g \setminus \inte(Y))$ the fibred product of the two 
 homomorphisms above, making the following diagram commutative: 
 \[ \begin{array}{ccc}
 \Mod^*(Y)\times _{S_{\partial Y}} \Mod^*(\Sigma_g \setminus \inte(Y)) & \to & \Mod^*(\Sigma_g \setminus \inte(Y)) \\
 \downarrow & & \downarrow \\
 \Mod^*(Y) & \to &  S_{\partial Y} \\
 \end{array}
 \]
 Assume now that $Y$ has genus $g$. 
 Then the stabilizer of the multicurve $\partial Y$ in $\Mod(\Sigma_g)$ is 
 $\Mod^*(Y)\times _{S_{\partial Y}} \Mod^*(\Sigma_g \setminus \inte(Y))$. 
 
 Further, the stabilizer in $\Bc_g$ of the multicurve 
$\partial X$ should be sent by $\phi$ into the stabilizer 
 of the multicurve $\partial Y$, where $Y=\Phi(X)$. 
Since $\Mod^*(X)$ stabilizes $\partial X$ we derive that 
 $\phi(\Mod^*(X))\subset \Mod^*(Y)\times _{S_{\partial Y}} \Mod^*(\Sigma_g \setminus \inte(Y))$.

 If $\pi_1$ and $\pi_2$ denote, respectively, the projections onto the first and second factors of the fibred product then 
 we derive a map $\pi_1\circ \phi: \Mod^*(X)\to \Mod^*(Y)$, such that 
 $\pi_1\circ \phi|_{\Mod(X)}$ is the conjugation $Ad_{\Phi}$ by $\Phi$. 

We claim that  $\pi_1\circ \phi=Ad_{\Phi}$. By composing with $Ad_{\Phi}^{-1}$ 
this reduces to show that if a homomorphism 
$\varphi:\Mod^*(X)\to \Mod^*(X)$  restricts to identity on $\Mod(X)$ then $\varphi$ is the identity. If $h\in \Mod^*(X)$ sends a boundary component $a$ into $b$, then 
\[ T_b=\varphi(T_{h(a)})=\varphi(hT_ah^{-1})=\varphi(h)T_a\varphi(h)^{-1}=T_{\varphi(h)a}\]
 so that $\varphi(h)a=b$. Thus the homomorphism $S_{\partial X}\to S_{\partial X}$ 
induced by $\varphi$ is identity. Then the Five Lemma  
implies that $\varphi$ is injective and Lemma \ref{lem:restriction} gives us the desired result, as the centralizer of $\Mod(X)$ in $\Mod^*(X)$ is trivial, if the 
 complexity of the surface $X$ is large enough.

 Further, the image $\pi_2\circ\phi(\Mod(X))$ is trivial, since $\phi(\Mod(X))=\Mod(Y)$. 
 Therefore $\pi_2\circ \phi$ factors through the quotient $\Mod^*(X)/\Mod(X)$, namely $S_{\partial X}$, and the map 
 $S_{\partial X}\to \Mod^*(\Sigma_g \setminus \inte(Y))\to S_{\partial Y}$ is the isomorphism induced by 
 the conjugation $Ad_{\Phi}$. This implies that there exists a section 
 $S_{\partial Y}\to \Mod^*(\Sigma\setminus Y)$ which determines an embedding of 
 $\Mod^*(Y)$ into $\Mod(\Sigma_g)$ so that $\phi(\Mod^*(X))$ is the image of $\Mod^*(Y)$. 
 \end{proof}

Continuing with the proof, observe that an immediate consequence of the proof of Theorem \ref{thm:fg} is that $\Bc_g$ has a system of generators $\mathcal S$ consisting  uniquely of elements belonging to $\Mod^*(X_i)$ for finitely many 
 $P$-suited subsurfaces $X_i\subset \Sigma_g$.  
 Lemma \ref{lem:Ad} shows that the natural map $\phi_{\Mod^*(X)}:\Mod^*(X)\to \Mod(\Sigma_g)$ coincides with ${Ad_{\Phi}}|_{\Mod^*(X)}$. In particular, $\phi(s)=Ad_{\Phi}(s)$, for 
 any $s\in S$.  Since both $\phi$ and $Ad_{\Phi}$ are homomorphisms, we derive that $\phi(s)=Ad_{\Phi}(s)$ 
for any  $s\in \Bc_g$. It follows that $Ad_{\Phi}(\Bc_g)\subseteq \Bc_g$,  
 namely $\Phi\in N_{\Mod(\Sigma_g)}(\Bc_g)$, as claimed.

To see that the homomorphism $N_{\Mod(\Sigma_g)}\to \Aut(\Bc_g)$ is injective, suppose that $\Phi \in \Mod(\Sigma_g)$ induces the identity automorphism of $\Bc_g$. Then $\Phi$ also induces the identity automorphism of $\PMod_c(\Sigma_g)$.
 At this point, Corollary \ref{cor:charac} tells us that  $\PMod_c(\Sigma_g)$ is normal in $\Bc_g$, since it is characteristic. In addition, $\PMod_c(\Sigma_g)$ has trivial centralizer in $\Mod(\Sigma_g)$, and therefore in $\Bc_g$, by Lemma \ref{lem:centralizer}. Hence, Lemma \ref{lem:restriction} implies that $\Phi$ is the identity. 
\end{proof}

\subsection{A note on normalizers}  We now explore some  properties of the normalizer of $\Bc_g$ (resp. $\Bh_g$) in $\Mod(\Sigma_g)$; as we will see, these represent a certain departure from the case of finite-type mapping class groups. First, we have:

\begin{lemma}
 Let $G$ be a subgroup of $\Mod(\Sigma_g)$. Assume that   
 $G\cap \PMod(\Sigma_g)$ is a normal subgroup of $\Mod(\Sigma_g)$. 
Then 
there is a natural surjective homomorphism
 \[N_{\Mod(\Sigma_g)}(G) \to N_{\Homeo(C)}(G/G\cap \PMod(\Sigma_g)).\]
 \label{lem:rubin}
 \end{lemma}
\begin{proof}
As $G\cap \PMod(\Sigma_g)$ is a normal subgroup of $G$ we derive from the exact sequence $(\ref{eq:SESgen})$ that  $G/G\cap \PMod(\Sigma_g)$ is a subgroup of 
$\Homeo(C)$. Then there is a well-defined homomorphism  
\[N_{\Mod(\Sigma_g)}(G) \to N_{\Homeo(C)}(G/G\cap \PMod(\Sigma_g))\]
that sends every element $H$ of the normalizer to its restriction to the space $C$  
of ends of $\Sigma_g$. 
Let now $h\in N_{\Homeo(C)}(G/G\cap \PMod(\Sigma_g))$.  By the exact sequence 
(\ref{eq:sesB})  $h$ admits a lift $H\in \Mod(\Sigma_g)$. 
Conjugation by $H$ sends the subgroup $G\cap \PMod(\Sigma_g)$ into itself, because it  
was assumed to be a normal subgroup of $\Mod(\Sigma_g)$. 
It follows that $H\in N_{\Mod(\Sigma_g)}(G)$, so that the homomorphism in the statement is surjective. 
\end{proof}
\begin{remark}
By a deep theorem of 
Rubin (see \cite{Rubin}), if $G/G\cap \PMod(\Sigma_g)$ is sufficiently large, in particular if 
$G\supseteq \Bc_g$, then there is an isomorphism 
\[\Aut(G/G\cap \PMod(\Sigma_g))\simeq N_{\Homeo(C)}(G/G\cap \PMod(\Sigma_g)).\]
\end{remark}\label{rem:rubin}
A recent result \cite{BCMNO} shows that the outer automorphism group $\Out(V)$ of $V$ is infinite. 
Combining this with the lemma above, we obtain: 
\begin{corollary}
The homomorphisms $\Out(\Bc_g)\to \Out(V)$ and $\Out(\Bh_g)\to \Out(V)$ are surjective. In particular, 
$\Out(\Bc_g)$ and $\Out(\Bh_g)$ are infinite. 
\label{cor:out2} 
\end{corollary}
\begin{proof}
From the exact sequences (\ref{eq:sesB}) and (\ref{eq:sesH}) the 
and Corollary \ref{cor:charac} $\Bc_g\cap \PMod(\Sigma_g)=\Bh_g\cap \PMod(\Sigma_g)=\PMod_c(\Sigma_g)$. 
Then Lemma \ref{lem:pmod_normal}  shows that the assumptions of Lemma \ref{lem:rubin} are satisfied. 
In particular we have surjective homomorphisms $N_{\Mod(\Sigma_g)}(\Bc_g)\to N_{\Homeo(C)}(V)$ 
and $N_{\Mod(\Sigma_g)}(\Bh_g)\to N_{\Homeo(C)}(V)$. By Rubin's Theorem  $N_{\Homeo(C)}(V)\simeq \Out(V)$. 
On the other hand   $N_{\Mod(\Sigma_g)}(\Bc_g)\subset \Out(\Bc_g)$ and 
 $N_{\Mod(\Sigma_g)}(\Bh_g)\subset \Out(\Bh_g)$. 
\end{proof}

We remark that a theorem of Ivanov \cite{Ivanov} mentioned above asserts that $\Out(\Mod(S_{g,n}))$ is always a finite group (in fact, trivial in all but finitely many cases). Thus Corollary \ref{cor:out} represents a limitation in the dictionary between asymptotic and finite-type mapping class groups.

Finally, we prove that the first homomorphism of Corollary \ref{cor:out} has infinite kernel: 

\begin{lemma}
For every $g\ge 0$, the homomorphism $\Out(\Bc_g)\to \Out(V)$ has infinite kernel. 
\end{lemma}\label{lem:infkernel}
\begin{proof} As usual, we prove the result for $\Bc_g$ only. 
Consider first the surface $\Sigma_0^1$, which recall that is homeomorphic to a sphere minus the union of a Cantor set and an isolated puncture. Using a totally analogous argument to that of Lemma \ref{lem:rubin}, we deduce that there is a surjective homomorphism 
\begin{equation}
\Out(\Bc_0^1)\to \Out(V)
\label{eq:ohmy}
\end{equation}
Recall the construction of the element $t_{\infty}$ from Example \ref{example}. 
Let $P$ be the pants decomposition underlying the rigid structure on $\Sigma_0^1$, and  $t_{\infty} \in \Mod(\Sigma_0^1)$ obtained as the product of all (say left) half-twists about the curves of $P$. 

We claim that $t_{\infty}\in N_{\Mod(\Sigma_0^1)}(\Bc_0^1)$. To see this, observe that if $Z$ is a defining surface  of an element $\varphi\in \Bc_g$, then each connected component of $\Sigma_0^1-Z$ is preserved by $t_{\infty}\varphi t_{\infty}^{-1}$; moreover, this element acts as the identity on every such component. In particular, $t_{\infty}\varphi t_{\infty}^{-1}\in \Bc_0^1$ and thus $t_{\infty}\in N_{\Mod(\Sigma_0^1)}(\Bc_0^1)$. Furthermore, the short exact sequence \eqref{eq:sesB} implies that $t_{\infty}^d\not\in \Bc_0^1$ unless $d=0$, since  $t_{\infty}^d$ fixes every end of $\Sigma_0^1$ but  does not have compact support unless $d=0$. Hence it provides an example of an infinite order element in 
$\Out(\Bc_0^1)$. On the other hand, its image under the homomorphism \eqref{eq:ohmy} is obviously trivial, and thus this homomorphism has infinite kernel.

After all this discussion, we treat the case of $\Sigma_g$. To do so, we simply embed $\Sigma_0^1$ in $\Sigma_g$ in such way that the pants decomposition $P$ underlying the rigid structure of $\Sigma_0^1$ is sent to the pants decomposition underlying the rigid structure of $\Sigma_g$. Using the same arguments as above, we deduce that the image of $t_\infty$ under the homomorphism induced by the embedding $\Sigma_0^1 \to \Sigma_g$ produces an infinite-order element in the kernel of the homomorphism $\Out(\Bc_g)\to \Out(V)$.
\end{proof}
%
%
%

\section{Surjections between asymptotic mapping class groups}
\label{sec:injective}

In this section we will prove Theorem \ref{thm:Ainj}. 

\begin{proof}[Proof of Theorem \ref{thm:Ainj}]
Let $0\le g<h$ and $h\geq 2$. Seeking a contradiction, suppose there were a surjective 
homomorphism  $\phi: \Bc_h \to \Bc_g$. 

For any $f\in \PMod_c(\Sigma_h)$ there exists, by Lemma \ref{lem:limit}, 
some $P$-suited subsurface $S_{h,n}\subset \Sigma_h$ such that $f$ is in the image of 
$\PMod(S_{h,n})$ within $\PMod_c(\Sigma_h)$. 
Let $a\in \phi(\PMod_c(\Sigma_h))$, so that $a=\phi(f)$, for some $f\in \PMod(S_{h,n})$. 
For any finite set $\{a_i\}\subset \Bc_g$ we can write $a_i=\phi(f_i)$, with 
$f_i\in \Bc_h$.  According to Lemma \ref{lemma-kapou}, 
there exists $f'\in \Bc_h$ such that $f$ belongs to 
the normal subgroup of $\Bc_h$ generated by $f'$ and such that $[f',f_i ff_{i}^{-1}]=1$.
This implies that $a$ belongs to the normal subgroup of $\Bc_g$ generated by 
$b=\phi(f')$ and $[b,a_i a a_i^{-1}]=1$. Lemma \ref{lemma-kapou} again implies that 
$a\in \PMod_c(\Sigma_g)$.

Restricting $\phi$ we obtain a homomorphism $\PMod(S_{h,n}) \to \PMod_c(\Sigma_g)$. Since $\PMod(S_{h,n})$ is finitely presented and $\PMod_c(\Sigma_g)$ is a direct limit of finitely presented groups, by Lemma \ref{lem:limit}, we deduce that there is $m\in \N$ and a nontrivial map $\PMod(S_{h,n}) \to \PMod(S_{g,m})$. But every such map has either 
a quotient of $\Z/10\Z$ as image, when $h=2$, or  trivial image, when $h\geq 3$, by \cite[Prop. 7.1]{AS} or \cite[Thm.7]{Castel}.

Let now $h\geq 3$. By the above $\phi(f)=1$, for any $f\in \PMod(\Sigma_h)$, and therefore  
$\phi$ factors through $V$. Let $\phi':V\to \Bc_g$ be the induced surjective homomorphism.  
If we compose the projection $p:\Bc_g\to V$ with $\phi'$  we obtain then 
a surjective homomorphism $V\to V$ with nontrivial kernel. Since $V$ is simple 
such a homomorphism should be trivial. This contradicts the surjectivity.

Assume that $h=2$. The nested groups $\phi(\PMod(S_{2,n}))$ are quotients of $\Z/10\Z$ for 
every $n$, in particular the sequence is eventually constant and $\phi(\PMod_c(\Sigma_{2}))$ 
is a quotient of $\Z/10\Z$. 
Therefore $\phi$ factors through the group $Q=\Bc_2/[\PMod_c(\Sigma_2), \PMod_c(\Sigma_2)]$, yielding a surjective homomorphism $\phi':Q\to \Bc_g$.
Observe that the restriction of the sequence \eqref{eq:SESgen} to $Q$ reads 
\[1\to  \Z/10\Z=H_1(\PMod_c(\Sigma_2))\to Q\to V\to 1.\]

Now, the normal subgroup $K=\phi'^{-1}(\PMod_c(\Sigma_g))\lhd Q$ is infinite 
because it is the inverse image of an infinite subgroup
by a surjective homomorphism. Therefore $K$ cannot be contained in the kernel of the projection homomorphism  $\overline{p}: Q\to V$, which is finite. Since $p(K)$ is a normal subgroup of $V$ and $V$ is simple,  
$\overline{p}(K)= V$. As $\ker \overline{p}$ is finite, we derive that $Q/K$ must be finite.
On the other hand, $K$ has infinite index in $Q$, because 
it is the preimage of the normal subgroup $\PMod_c(\Sigma_g)$ of $\Bc_g$ by the surjective 
homomorphism $\phi'$. 
Hence we get  the desired contradiction. 
\end{proof}

\begin{remark}
The same reasoning as above along with Remark \ref{H-kapou} 
implies that there are no surjective homomorphisms $\Bh_h \to \Bh_g$ (resp. $\Bh_h \to \Bc_g$
and $\Bc_h \to \Bh_g$) if $g<h$.
 \end{remark}


\section{Homomorphisms from lattices and Property (T)}
\label{sec:lattices}
We start this section by proving Theorem \ref{thm:T}, which is a direct consequence of the existence of the short exact sequence \eqref{eq:sesB}.

Before we prove the result, recall that a discrete group $G$ is said to have {\em Kazhdan's  Property (T)} if every action of $G$ by continuous affine isometries on a real Hilbert space has a fixed point. On the other hand,  $G$ has the {\em Haagerup property} if it admits a proper action by continuous affine isometries on a Hilbert space. It follows that if $G$ has both properties, then $G$ is finite. We refer the reader to the book \cite{BHV} for a thorough exposition of these properties. 
We may now prove Theorem \ref{thm:T}, whose statement we now recall: 

\begin{named}{Theorem \ref{thm:T}}
Let $g\ge 0$. Then $\Bc_g$ and $\Bh_g$ do not have Property (T). 

\end{named}

\begin{proof}[Proof of Theorem \ref{thm:T}]
Recall from equation \eqref{eq:sesB} that there is a surjective homomorphism $\Bc_g \to V$. It is known (see, for instance, \cite[Theorem 1.7.1]{BHV}) that Property (T) is preserved under quotients. On the other hand, a result of Farley \cite{Far} asserts that $V$ has the Haagerup property. Since $V$ is infinite,  it follows that $V$ does not have Property (T), and thus neither does $\Bc_g$.  The same argument yields the result for $\Bh_g$. 
\end{proof}

We now proceed to give a proof of Theorem \ref{thm:lattice}. The proof is again based on properties of the group $V$, plus 
a result of Farb-Masur \cite{FM} which asserts that every homomorphic image of a lattice in a finite-type mapping class group is finite. More concretely, they proved:

\begin{theorem}[\cite{FM}]
Let $\Gamma$ be an irreducible lattice in a semisimple Lie group of real rank at least two. For every $g, n\ge 0$, every homomorphism $\Gamma \to \Mod(S_{g,n})$ has finite image. 
\label{thm:FM}
\end{theorem} 

Before giving a proof of Theorem \ref{thm:lattice}, we remind the reader of its statement:

\begin{named}{Theorem \ref{thm:lattice}}
Let $\Gamma$ be a lattice in a  semisimple Lie group $G$ of real rank at least two, where $G$ has no compact factors isogenous to ${\rm SU}(1,n)$ or ${\rm SO}(1,n)$. 
Then, any homomorphism from $\Gamma$ to $\Bc_g $ (resp. $\Bh_g$) has finite image. 
\end{named}

\begin{proof}
Let $G$ be a Lie group as in the statement. As such, $G$ has Property (T) (see \cite[Theorem 13.2.4]{Witte}, for instance), and therefore so does $\Gamma$ (see \cite[Theorem 1.7.1]{BHV}). For such a lattice $\Gamma$, let $\phi: \Gamma \to \Bc_g$ be a homomorphism. Recall again the short exact sequence 
 \eqref{eq:sesB}: \[ 1\to \PMod_c(\Sigma_g)\to \Bc_g\stackrel{p}{\to} V\to 1.\] 
 As in the proof of Theorem \ref{thm:T}, the fact that $V$ has the Haagerup property implies that $(p \circ \phi)(\Gamma)$ is finite. But $\phi(\Gamma) \cap \ker(p)$ is also finite. To see this, observe that since $\Gamma$ is finitely presented (see \cite[Theorem 4.7.10]{Witte}), Lemma \ref{lem:limit} implies that there exists a finite-type subsurface $S$ of $\Sigma_g$ such that $\phi(\Gamma) \cap \ker(p)$ is contained in $\PMod(S)$. At this point, Theorem \ref{thm:FM} tells us that $\phi(\Gamma) \cap \ker(p)$ is also finite, as desired. 
 
 We obtain the result for $\Bh_g$ in an analogous way. 
 \end{proof}


\section{Non-linearity}

In this section we will prove Theorem \ref{thm:nonlinear}. As mentioned in the introduction, we will do so by showing that $\Bc_g$ contains a copy of {\em Thompson's group $F$}. Recall that $F$ is the group of piecewise-linear self-homeomorphisms of $[0,1]$ that preserve rational dyadic numbers, are differentiable except at finitely many dyadic rationals, and at each interval of differentiability the slopes are  powers of 2. 

A strongly related group which we will also need is {\em Thompson's group $T$}, whose definition is the same as that of $F$ but with the unit circle $\mathbb{S}^1$ instead of $[0,1]$. One has the well-known inclusions $F\subset T \subset V$; for a proof, as well as a detailed discussion on Thompson's groups, see \cite{CFP}. We now prove: 

\begin{proposition}\label{split}
Let $g\ge 0$. The short exact sequence \eqref{eq:sesB}, that is \[1 \to \PMod_c(\Sigma_g) \to  \Bc_g \to V \to 1,\] splits over  Thompson's group $F$.
\end{proposition} 
\begin{proof} 
The result is known for $g=0$ (see \cite{KS,FK2}), where in addition it is shown that the sequence \eqref{eq:sesB} splits over  Thompson group's $T$, identified with the subgroup of $\Bc_0$ preserving the whole visible side of $\Sigma_0$.

Consider now a closed disk $\Sigma_{0,1}$ with a Cantor set removed from its interior. The surface $\Sigma_{0,1}$ comes equipped with an obvious rigid structure that comes from  that of $\Sigma_0$ under the natural subsurface embedding $\Sigma_{0,1}\hookrightarrow\Sigma_0$. 
The group of asymptotically rigid mapping classes of $\Sigma_{0,1}$ is the pure braided Thompson group 
$BV$, a subgroup of the braided Thompson group considered by Brin \cite{Brin} and Dehornoy (\cite{Deh}). 
When $g>0$ 
the map $\Mod(\Sigma_{0,1})\to \Mod(\Sigma_g)$ is injective. As a consequence $BV$ injects into 
$\Bc_g$. On the other hand it is known that $BV$ splits over $F$ (see \cite{Brin,FK}). 
\end{proof}

Finally, we prove Theorem \ref{thm:nonlinear}:

\begin{proof}[Proof of Theorem \ref{thm:nonlinear}]
 By Proposition \ref{split}, $B_g$ contains an isomorphic copy of Thompson's group $F$. Now, $F$ has simple commutator subgroup \cite{CFP}, and in particular is not residually finite. Since finitely-generated subgroups of linear groups are residually finite and $F$ is finitely generated, we deduce that $\Bc_g$ (and thus $\Bh_g$) is not linear. 
\end{proof}

\bigskip

 \end{document}